\documentclass[12pt]{article}

\usepackage[utf8]{inputenc}

\usepackage{amssymb,amsmath,amsthm,url,graphicx,wrapfig}

\usepackage{enumerate}

\theoremstyle{plain}
\newtheorem{thm}{Theorem}
\newtheorem{lm}{Lemma}
\newtheorem{clm}{Claim}

\usepackage{todonotes}

\title{Covering three-tori with cubes}

\author{Ilya I. Bogdanov\footnote{Moscow Institute of Physics and Technology, Laboratory of Combinatorial and Geometric Structures, Moscow, Russia}
\and
Oleg Grigoryan\footnote{Yandex School of Data Analysis, Moscow, Russia}%
\and
Maksim Zhukovskii\footnotemark[\value{footnote}]
}

\date{}

\begin{document}



\maketitle

\begin{abstract}
Let $\mu(\varepsilon)$ be the minimum number of cubes of side $\varepsilon$ needed to cover the unit three-torus $[\mathbb{R}/\mathbb{Z}]^3$. We prove new lower and upper bounds for $\mu(\varepsilon)$ and find the exact value for all $\varepsilon\geq\frac{7}{15}$ and all $\varepsilon\in\left[\frac{1}{r+1/(r^2+r+1)},\frac{1}{r-1/(r^2-1)}\right)$ for any integer $r\geq 3$.
\end{abstract}

\maketitle

\section{Introduction}

Let $d$ be a positive integer, and let $\varepsilon\in(0,1)$. Consider the torus $T^d:=[\mathbb{R}/\mathbb{Z}]^d$ and the set $\mathcal{J}_{\varepsilon}$ of `sub-cubes' of the form $\{(x_1,\ldots,x_d):\,x_i\in[x^0_i,x^0_i+\varepsilon]\}$. The question is, what is the minimum number $\mu:=\mu(d;\varepsilon)$ of sets $A_1,\ldots,A_{\mu}$ from $\mathcal{J}_{\varepsilon}$ needed to cover $T^d$ (i.e., $T^d=A_1\cup\ldots\cup A_{\mu}$)?

In~\cite{2-dimension}, it is proven that
\begin{equation}
\mu\geq \lceil 1/\varepsilon\rceil^{(d)},
\label{general_lower_bound}
\end{equation}
where $\lceil x\rceil^{(i)}=\left\lceil x\lceil x\rceil^{(i-1)}\right\rceil$ and $\lceil x\rceil^{(1)}=\lceil x\rceil$. Moreover, it is shown that, for $d=2$, this lower bound is sharp, i.e. $\mu(2;\varepsilon)=\bigl\lceil \frac1\varepsilon\bigl\lceil \frac1\varepsilon\bigl\rceil\bigl\rceil$.

In our paper, we consider $d=3$. Since $\mu(1;\varepsilon)=\bigl\lceil \frac1\varepsilon\bigl\rceil$ and $\mu(2;\varepsilon)=\bigl\lceil \frac1\varepsilon\bigl\lceil \frac1\varepsilon\bigl\rceil\bigl\rceil$, we get that
\begin{equation}
  \left\lceil \frac1\varepsilon\left\lceil \frac1\varepsilon\left\lceil \frac1\varepsilon\right\rceil\right\rceil\right\rceil\leq\mu(3;\varepsilon)\leq\left\lceil \frac 1\varepsilon\left\lceil \frac1\varepsilon\right\rceil\right\rceil\cdot \left\lceil \frac 1\varepsilon\right\rceil.
\label{known-bounds}
\end{equation}
In~\cite{2-dimension}, the authors also noticed that the lower bound
~
in~\ref{known-bounds}
for $d=3$ is not sharp. For example, $\mu\bigl(3;\frac37\bigr)>\bigl\lceil \frac73\bigl\lceil \frac73\bigl\lceil \frac73\bigr\rceil\bigr\rceil\bigr\rceil$.
Unfortunately, no general bound better than~(\ref{known-bounds}) is known.

The case of large $d$ was studied more extensively. Note that (\ref{general_lower_bound}) implies $\mu(d;\varepsilon)\geq (1/\varepsilon+o(1))^d$. From a general result by Erd\H{o}s and Rogers~\cite{ErdosRogers} it follows that $1/\varepsilon$ is the correct base of the exponent, i.e. $\mu(d;\varepsilon)=(1/\varepsilon+o(1))^d$. More precisely, they proved that $\mu(d;\varepsilon)=O(d\log d(1/\varepsilon)^d)$. In the discrete case (i.e. $[\mathbb{Z}/t\mathbb{Z}]^d$ is covered by cubes with side $s\in[t]:=\{1,2,\dots,t\}$, $\varepsilon=s/t$) a direct application of the probabilistic method allows to put away the $\log d$ factor (see~\cite{Bollobas}), and it is easy to see that it implies the same bound for any $\varepsilon$ (both rational and irrational): $\mu(d;\varepsilon)=O(d(1/\varepsilon)^d)$ while (\ref{general_lower_bound}) gives $\mu(d;\varepsilon)=\Omega((1/\varepsilon)^d)$. Note that, as we show in Section~\ref{integer_lattices}, the discrete and continuous problems are in some sense equivalent. 
 Let us also mention that, for $\varepsilon=\frac{2}{r}$, $r\in\mathbb{N}$, the corresponding packing problem (finding $\nu(d;\varepsilon)$ --- the maximum number of non-overlapping sub-cubes with side $\varepsilon$ inside $T^d$) is related to the problem of finding Shannon capacity $c(C_r)$ of a simple cycle on $r$ vertices~\cite{Shannon}: $c(C_r)=\sup_{d\geq 1}(\nu(d;2/r))^{1/d}$ (the same connection works for other rational $\varepsilon$ but the respective graphs are not so foreseeable). For even $r$, $c(C_r)=r/2$ is straightforward.  It was shown by Lov\'{a}sz~\cite{Lovasz} that $c(C_5)=\sqrt{5}$. For larger odd $r$, finding $c(C_r)$ is still open.

In this paper, we have found the exact value of $\mu(3;\varepsilon)$ for $\varepsilon\geq 7/15$. We have also found exact values of $\mu(3;\varepsilon)$ for $\varepsilon$ close to $1/r$, $r\in\mathbb{N}$. In Section~\ref{new}, we state new results. 

\smallskip
For convenience, for any $a,b\in \mathbb R/\mathbb Z$, we denote by $|a-b|$ the smallest $\nu\in[0,1)$ such that $a=b\pm\nu$. Similarly, for every $a,b\in\mathbb Z/t\mathbb Z$, we denote by $|a-b|$ the smallest $\nu\in\{0,1,\dots,t-1\}$ such that $a=b\pm\nu$.

\section{New results}
\label{new}



Since, for an integer $r\geq 2$ and $\varepsilon\in\left[\frac{1}{r},\frac{1}{r-1/r^2}\right)$, the lower bound and the upper bound in (\ref{known-bounds}) are equal, the value of $\mu(3;\varepsilon)$ is straightforward and equals $r^3$. We have also found left-neighborhoods of all $1/r$ where the lower bound 
in~\eqref{known-bounds} is tight
(notice that for such $\varepsilon$ the difference between the upper and the lower bounds is, conversely, large).

\begin{thm}
Let $r\in\mathbb{N}$. If $\varepsilon\in\left[\frac{1}{r+ 1/(r^2+r+1)},\frac{1}{r}\right)$, then the lower bound is tight, i.e. $\mu(3;\varepsilon)=r^3 + r^2 + r + 1$.
\label{left_interval}
\end{thm}

Moreover, we have proved that the trivial right-neighborhoods of any number of the form $1/r$ where the upper bound is tight can be extended in the following way.

\begin{thm}
Let $r\geq 2$ be an integer. If $\varepsilon\in\left[\frac{1}{r},\frac1{r-1/(r^2-1)}
\right)$, then the upper bound is tight, i.e. $\mu(3;\varepsilon)=r^3$.
\label{right_interval}
\end{thm}

We have also improved the lower bound from~(\ref{known-bounds}) in some special cases.

\begin{thm}
Let $r\geq 2$ be an integer, $\xi\in[r]$ be such that
\begin{equation}
  \label{xi^2}
  \xi^2\leq\xi+(r+1)\left\lfloor\frac{\xi^2}{r+1}\right\rfloor.
\end{equation}
Let
\begin{itemize}
\item[$\bullet$] $s=r^2+r+\xi$,

\item[$\bullet$] $t=r^3+r^2+2\xi r+\left\lfloor\frac{\xi^2}{r+1}\right\rfloor$.
\end{itemize}
Assume that $t$ and $s$ are coprime.
Then $\mu\left(3;\frac{s}{t}\right) > t$, i.e. bigger than the lower bound.
\label{lower_bounds}
\end{thm}

Condition
~\eqref{xi^2} implies that either $\xi\geq\sqrt{r+1}$  or $\xi = 1$. In the interval $[1/3,1/2)$, there are two such
values of $\frac st$, namely $\frac{s}{t}\in\left\{\frac{7}{16},\frac{8}{21}\right\}$.

Finally, we have improved the upper bound from~(\ref{known-bounds}) in some special cases.

\begin{thm}
Let $r\geq 2$ be an integer, $\xi\in[r]$. Let
\begin{itemize}
    \item[$\bullet$] $s=r^2+r+\xi$,
    \item[$\bullet$] $t=r^3+r^2+\xi(r+1)$.
\end{itemize}
Then $\mu\left(3;\frac{s}{t}\right) \leq t$, i.e. smaller than the upper bound.
\label{upper_bounds}
\end{thm}

Notice that Theorem~\ref{left_interval}
follows from Theorem~\ref{upper_bounds} for $\xi=1$. In Section~\ref{th5_proof}, we give a complete proof of Theorem~\ref{upper_bounds} for all values of $\xi$.

Both Theorem~\ref{lower_bounds} and Theorem~\ref{upper_bounds} imply improvements of the bounds in~(\ref{known-bounds}) for values of $\varepsilon$ in certain right-interval of the respective $s/t$ (see Lemma~\ref{lm1} from Section~\ref{integer_lattices}).

The results above allow to find the values of $\mu(3;\varepsilon)$ for all $\varepsilon\in[1/2,1)$ as follows.

\begin{thm}
We have
$$
  \mu(3;\varepsilon)=\begin{cases}
    4, & \varepsilon\in[3/4,1);\\
    5, & \varepsilon\in[2/3,3/4);\\
    7, & \varepsilon\in[3/5,2/3);\\
    8, & \varepsilon\in[1/2,3/5).
  \end{cases}
$$
\label{first_interval}
\end{thm}

Notice that, in contrast to $d=2$, when $\varepsilon\in[1/2,1)$, the value of $\mu(d=3;\varepsilon)$ achieves the lower bound in~(\ref{known-bounds}) if and only if
$\varepsilon\in[1/2,4/7)\cup[3/5,1)$.

\smallskip
For $\varepsilon\in[1/3,1/2)$, Theorems~\ref{left_interval},~\ref{right_interval},~\ref{lower_bounds},~\ref{upper_bounds} and Lemma~\ref{lm1} from Section~\ref{integer_lattices} imply that
\begin{itemize}
\item for $\varepsilon \in \left[\frac{1}{3}, \frac{8}{23}\right)$, $\mu(3;\varepsilon)=27$ (by Theorem~\ref{right_interval});
\item for $\varepsilon \in \left[\frac{8}{21},\frac{5}{13}\right)$,
  $\mu(3;\varepsilon)$ is constant and lies on $[22,24]$ (by Theorem~\ref{lower_bounds}, Lemma~\ref{lm1}, and \eqref{known-bounds});
\item for $\varepsilon \in \left[\frac{7}{16},\frac{4}{9}\right)$,
  $\mu(3;\varepsilon)$ is constant and lies on $[17,21]$ (by Theorem~\ref{lower_bounds}, Lemma~\ref{lm1}, and \eqref{known-bounds});
\item for $\varepsilon \in \left[\frac{4}{9},\frac{7}{15}\right)$, $\mu(3;\varepsilon)\in [16,18]$ (by Theorem~\ref{upper_bounds} and \eqref{known-bounds}
    );
\item for $\varepsilon \in \left[\frac{7}{15}, \frac{1}{2}\right)$, $\mu(3;\varepsilon) = 15$ (by Theorem~\ref{left_interval} and Lemma~\ref{lm1}).
\end{itemize}

\section{Integer lattices}
\label{integer_lattices}

On one hand, there are continuously many $\varepsilon$ left for which the answer is not known. On the other hand, the lemma below implies that the problem reduces to a countable set.

Let $d\geq 2$ be an integer.

\begin{lm}
There exists an infinite sequence of rational numbers
$1>\frac{s_1}{t_1}>\frac{s_2}{t_2}>\ldots >0$ such that, for every $i\in\mathbb{N}$, $t_i\leq\mu(d;s_i/t_i)$ and $\mu(d;\varepsilon)=\mu(d;s_i/t_i)$ for all $\varepsilon\in[s_i/t_i,s_{i-1}/t_{i-1})$, where $s_0=t_0=1$.
\label{lm1}
\end{lm}

For $d=3$, due to~(\ref{known-bounds}), the denominator of the {\it critical point} $\frac{s_i}{t_i}$ is at most $\lceil \frac{t_i}{s_i} \rceil\lceil \frac{t_i}{s_i}\lceil \frac{t_i}{s_i} \rceil \rceil\leq \lceil \frac{t_i}{s_i}\rceil^3
$. Therefore, for every integer $r\geq 2$, on $\left[\frac{1}{r},\frac{1}{r-1}\right)$ there are at most $\frac{r^2(r^3+1)}{2(r-1)}+r^3$ candidates on the role of a critical point.

We give the proof of Lemma~\ref{lm1} in Section~\ref{integers}.
The proof also yields that the problem can be equivalently reformulated for integer lattices (as stated below in Lemma~\ref{lm2}).

Let $s\leq t$ be positive integers. Consider the torus $[\mathbb{Z}/t\mathbb{Z}]^d$ and the set of its `sub-cubes' with edges of size $s$: $\{(x_1,\ldots,x_d):\,x_i^0\leq x_i\leq x_i^0+s-1\mod t\}$, $x_i^0\in \mathbb{Z}/t\mathbb{Z}$. Throughout the paper, for a `sub-cube' $\{(x_1,\ldots,x_d):\,x_i^0\leq x_i\leq x_i^0+s-1\mod t\}$, we call its node $(x_1^0,\dots,x_d^0)$ {\it the base vertex of the cube}. Let $\mu_0(d;s,t)$ be the minimum number of such `sub-cubes' needed to cover the torus.

\begin{lm}
Let $r\geq 2$ be an integer, $\varepsilon\in\left[\frac{1}{r},\frac{1}{r-1}\right)$. Let $\frac{s}{t}\leq\varepsilon$ be the (leftwards)
closest rational number to $\varepsilon$ with $t\leq r^d$. Then $\mu(d;\varepsilon)=\mu_0(d;s,t)$.
Moreover, for any rational $\varepsilon=s/t$, we have $\mu(d;\varepsilon)=\mu_0(d;s,t)$.
\label{lm2}
\end{lm}

Since $\mu(d;s/t)=\mu_0(d;s,t)$, we get that $\mu_0(d;s,t)$ depends only on $s/t$.


\section{Proofs of Lemmas~\ref{lm1} and~\ref{lm2}}
\label{integers}

In this section, we prove the following.

\begin{clm}
  \label{clm1}
  Let $\mu$ be a positive integer. Choose a minimal $\varepsilon_0$ such that there is a covering $\mathcal{A}$ of the torus $T^d$ by $\mu$ `$\varepsilon_0$-sub-cubes'
  $$
    A_j:=\left\{(x_1,\ldots,x_d):\,x_i\in\left[x_i^j,x_i^j+\varepsilon_0\right]\right\},\quad j\in[\mu].
  $$
  Then $\varepsilon_0=s_0/t_0$ where $s_0$ and $t_0$ are coprime integers, with $t_0\leq\mu$.

  Moreover, for every rational $\varepsilon=s/t\geq\varepsilon_0$, there exists a covering of $T^d$ with $\mu$ `$\varepsilon$-sub-cubes' such that all their base vertices are multiples of $1/t$.
\end{clm}

Lemma~\ref{lm1} is a direct consequence of the first part of Claim~\ref{clm1}. To show that Claim~\ref{clm1} also yields Lemma~\ref{lm2}, take an $\varepsilon\in\left[\frac{1}{r},\frac{1}{r-1}\right)$ and put $\mu=\mu(d;\varepsilon)$; notice that $\mu\leq r^d$ and choose the fraction $s/t$ as in the statement of Lemma~\ref{lm2}. By Claim~\ref{clm1}, the minimal number $\varepsilon_0$ such that there is a covering of the torus $T^d$ by $\mu$ `$\varepsilon_0$-sub-cubes' has the form $\varepsilon_0=s_0/t_0$ with $t_0\leq\mu\leq r^d$; therefore, $\varepsilon_0\leq s/t\leq\varepsilon$ and hence $\mu(d;\varepsilon_0)=\mu(d;s/t)=\mu$. Moreover, assuming that $\varepsilon=s'/t'$ is rational, the second part of Claim~\ref{clm1} implies that there exists a covering of $T^d$ with $\mu$ `$\frac {s'}{t'}$-sub-cubes' such that all their base vertices are integer multiples of $1/t$. This covering induces a covering of $(\mathbb Z/t\mathbb Z)^d$ by $\mu$ `sub-cubes' of side length $s$, thus showing that $\mu_0(d;s',t')\leq \mu(d;s'/t')$. This yields Lemma~\ref{lm2}, as the converse inequality $\mu_0(d;s',t')\geq \mu(d;s'/t')$ is trivial.

The remaining part of the section is devoted to the proof of Claim~\ref{clm1}.

\smallskip
Notice that the minimal $\varepsilon_0$ chosen in the statement of Claim~\ref{clm1} exists by a standard compactness argument. Consider now any $\varepsilon$ such that there is a covering $\mathcal{A}$ of the torus $T^d$ by $\mu$ `$\varepsilon$-sub-cubes'
$$
A_j:=\left\{(x_1,\ldots,x_d):\,x_i\in\left[x_i^j,x_i^j+\varepsilon\right]\right\},\quad j\in[\mu].
$$
We will modify this covering in two steps. The result of Step~1 will, in particular, establish the second part of Claim~\ref{clm1}. In Step~2, we will see that, if $\varepsilon$ is not of the form $\varepsilon=s/t$ with $t\leq \mu$, then there is a covering with $\mu$ `$\varepsilon'$-sub-cubes' for some $\varepsilon'<\varepsilon$, thus establishing the first part of Claim~\ref{clm1}.

\smallskip
Let $i\in[d]$. Consider the segments $\left[x_i^j,x_i^j+\varepsilon\right]$, $j\in[\mu]$. Let us introduce a graph $G_i(\mathcal{A})$ with vertex set $[\mu]$. Let vertices $j_1,j_2\in[\mu]$ be adjacent in $G_i(\mathcal{A})$ if and only if the sets $\left\{x_i^{j_1},x_i^{j_1}+\varepsilon\right\}$, $\left\{x_i^{j_2},x_i^{j_2}+\varepsilon\right\}$ are not disjoint (i.e., the respective segments have at least one common endpoint).

\smallskip
\textit{Step 1.} We show that there exists a covering $\widetilde{\mathcal A}$ of $T^d$ by $\mu$ `$\varepsilon$-sub-cubes' such that $G_i(\widetilde {\mathcal A})$ is connected, for every $i\in[d]$. For that purpose, we shift some `sub-cubes' as follows.

Assume that, in $G_i(\mathcal{A})$, there are several connected components $H_1,\ldots,H_{\ell}$, $\ell\geq 2$.
Choose an endpoint $a$ of segment $j$ with $j\in H_1$ and an endpoint $b$ of segment $j'\in [\mu]\setminus H_1$ (i.e., $a\in\{x_i^j,x_i^j+\varepsilon\}$ and $b\in\{x_i^{j'},x_i^{j'}+\varepsilon\}$) such that their distance $\rho=|a-b|$ is minimal over the choice of the components, of the segments inside them, and their endpoines; since $j$ and $j'$ are in different components, we have $\rho>0$.

Without loss of generality, $a>b$. Now let us shift all segments labeled by $H_1$ leftwards by distance $\rho$. By the choice of $\rho$, no point may remain uncovered, so we get a covering $\mathcal{A}_1$ of $T^d$, where the graph $G_i(\mathcal{A}_1)$ consists of at most $\ell-1$ components. If $G_i(\mathcal{A}_1)$ is not connected, we perform the same procedure with $\mathcal{A}_1$ and obtain a covering $\mathcal{A}_2$ with $G_i(\mathcal{A}_2)$ having at most $\ell-2$ components. Proceeding in this way, we reach a covering $\widehat{\mathcal A}$ where $G_i(\widehat{\mathcal A})$ is connected.

Applying the same procedure for every $i\in[d]$, we get a desired covering $\widetilde{\mathcal A}$.

\smallskip
Now, if $\varepsilon=s/t$ is rational, we may assume that that $x_1=(0,0,\dots,0)$. By connectedness of all graphs $G_i(\widetilde{\mathcal{A}})$, all coordinates of the vertices of the `sub-cubes' are multiples of $1/t$; thus the second part of Claim~\ref{clm1} is established.


\smallskip
\textit{Step 2.}
Now we may assume that for every $i\in[d]$, the graph $G_i(\mathcal{A})$ is connected. Suppose that there is no integer $t \in[\mu]$ such that $t\varepsilon\in \mathbb{N}$; we will show that there exists an $\varepsilon'<\varepsilon$ such that $T^d$ can be covered by $\mu$ `$\varepsilon'$-sub-cubes', thus proving the first part of Claim~\ref{clm1}.

Fix $i\in[d]$ and consider the following relation $<_i$ on the set of segments $[x_i^j,x_i^j+\varepsilon]$:
\begin{center}
if $x_i^{j_1}+\varepsilon=x_i^{j_2}\mod 1$, then
$[x_i^{j_1},x_i^{j_1}+\varepsilon]<_i[x_i^{j_2},x_i^{j_2}+\varepsilon].$
\end{center}
Since $q\varepsilon\notin\mathbb{N}$ for any $q \in[\mu]$, we get that $<_i$ is a
partial order (recall that some numbers of the form $x_i^j$ may coincide).

Choose now a minimal segment $[x_i^{j_i},x_i^{j_i}+\varepsilon]$ with respect to $<_i$. Shifting all $i$th coordinates by $x_i^{j_i}$, we may assume that $x_i^{j_i}=0$. Performing such shifts along all coordinates, we arrive at the situation where all coordinates of all vertices of the `sub-cubes' lie in the set
$$
  I_\varepsilon=\bigl\{ k\varepsilon\mod 1\colon k=0,1,\dots,\mu\bigr\}.
$$
Notice that $I_\varepsilon$ consists of $\mu+1$ distinct numbers; let $\gamma$ be the minimal distance between elements of $I_\varepsilon$ (modulo~$1$).

Choose a positive $\delta<\gamma/\mu$ and $\varepsilon'=\varepsilon-\delta$. Then, for every $k_1,k_2\in\{0,1,\dots,\mu\}$ we have
$$
  \{k_1\varepsilon\}<\{k_2\varepsilon\} \quad\text{if and only if} \quad
  \{k_1\varepsilon'\}<\{k_2\varepsilon'\}.
$$
Innformally speaking, the sets $I_\varepsilon$ and $I_{\varepsilon'}$ have the same combinatorial structure.

Now choose the `$\varepsilon'$-sub-cubes'
$$
  A_j'=\left\{(x_1,\ldots,x_d):\,x_i\in\left[(x')_i^j,(x')_i^j+\varepsilon'\right]\right\},\quad j\in[\mu],
$$
as follows: if $x_i^j=\{k\varepsilon\}$, then $(x')_i^j=\{k\varepsilon'\}$. The above relation shows that for every $a'\in[0,1)$ there exists $a\in[0,1)$ such that
$$
  a'\in[(x')_i^j,(x')_i^j+\varepsilon'] \quad\text{if and only if}\quad a\in[x_i^j,x_i^j+\varepsilon].
$$
Therefore, for any point $x'=(x_1',\dots,x_d')\in T^d$ there exists a point $x=(x_1,\dots,x_d)\in T^d$ such that $x'$ is covered by $A_j'$ if and only if $x$ is covered by $A_j$. Thus, since the $A_j$ form a covering of $T^d$, so do the $A_j'$, and we have constructed a covering by $\mu$ `$\varepsilon'$-sub-cubes', as desired.

\section{Proofs of Theorems~\ref{right_interval} and~\ref{lower_bounds}}
\label{proof_lower_bound}

Introduce the following conditions on positive integers~$s$ and~$t$:
\begin{enumerate}
  \renewcommand\theenumi{(\roman{enumi})}
    \item\label{one} $\left\lfloor\frac{t}{s}\right\rfloor=r\geq 2$ and $\left\lceil\frac{t}{s}\right\rceil=r+1$,
    \item\label{two} $\left\lceil\frac{t}{s}\left\lceil\frac{t}{s}\right\rceil\right\rceil = s$ and, consequently, $\left\lceil\frac{t}{s}\left\lceil\frac{t}{s}\left\lceil\frac{t}{s}\right\rceil\right\rceil\right\rceil = t$,
    \item\label{4_ineq_main} $ (s^2 - t(r+1))r \leq s - (s^2 - t(r+1))$.
    \item\label{5_ineq_main} $ r(s^2 + s - t(r+1)) \leq t$;
    \item\label{4 or 5} at least one of the inequalities in~\ref{4_ineq_main} and~\ref{5_ineq_main} is strict.
\end{enumerate}

In fact, Theorems~\ref{right_interval} and~\ref{lower_bounds} are particular cases of the following lemma. We start with proving the Lemma, and then we derive both theorems from it.

\begin{lm} If $t,s$ are coprime positive integers satisfying the conditions \ref{one}--\ref{4 or 5}, then $\mu(3;s/t)>t$.
\label{Lemma_for_lower_bounds}
\end{lm}

\begin{proof} Due to Lemma~\ref{lm2}, it suffices to show $\mu_0(3;s,t)>t$. By the indirect assumption, there is a covering of $(\mathbb Z/t\mathbb Z)^3$ with $t$ `sub-cubes'
$C_1,C_2, \ldots, C_t$ of side~$s$.

For $i\in[t]$,
let 
$\mathbf{x}^0(i)=(x^0_1(i),x^0_2(i),x^0_3(i))$ be the base vertex of cube $C_i$. For every $\alpha\in\mathbb Z/t\mathbb Z$, define the \emph{$\alpha$th layer $\mathcal S_\alpha$} as the intersection of the torus with the hyperplane $x_3=\alpha$, i.e., $\mathcal{S}_{\alpha}=(\mathbb{Z}/t\mathbb{Z})^3\big|_{x_3=\alpha}$.
Each such intersection is a 2-torus covered with the \emph{squares}, i.e., the (nonempty) intersections of the `sub-cubes' with $\mathcal S_\alpha$.

Let us denote the number of squares covering layer $\mathcal{S}_\alpha$ by $f(\alpha)$.
Since $\mu(2;\varepsilon)=\bigl\lceil \frac 1\varepsilon\bigl\lceil \frac1\varepsilon\bigr\rceil\bigr\rceil$ (see~\cite{2-dimension}), we get $f(\alpha)\geq s$ for every~$\alpha$ due to \ref{two}.
On the other hand, $\sum_{\alpha\in\mathbb Z/t\mathbb Z} f(\alpha) = st$  because each `sub-cube' meets exactly $s$ layers. This means that $f(\alpha)=s$ for every $\alpha\in\mathbb Z/t\mathbb Z$. In particular, for every $i\in[t]$ we have $f(x^0_3(i)+s-1)=f(x^0_3(i)+s)$. In other words, the last layer of the cube $C_i$ is covered by the same number of squares as the next layer. It is only possible when there is $j\in[t]$ such that $x^0_3(j)=x^0_3(i)+s$;
informally speaking, cube $C_j$ `starts' exactly when $C_i$ `finishes'. In this situation, say that cube~$C_j$ is a \emph{successor} of cube~$C_i$.

Since $s$ and $t$ are coprime, for every $\alpha\in\mathbb Z/t\mathbb Z$, there is a cube $C_i$ with $x^0_3(i)=\alpha$. 
Renumbering the cubes $C_i$, we assume that $x_3^0(i)=i$ for all $i\in[t]$; we assume that the numeration of the cubes is cyclic, i.e., $C_{i+t}=C_i$; so, further we assume that the indices $i$ run over $\mathbb Z/t\mathbb Z$. Now, each cube $C_i$ has a unique successor $C_{i+s}$.


For $\gamma\in\{1,2,3\}$, we say that a \emph{$\gamma$-column} is a set of $t$ elements in $(\mathbb Z/t\mathbb Z)^3$ differing only in the $\gamma$th coordinate. Due to~\ref{one}, each $\gamma$-column meets at least $r+1$ `sub-cubes'.

\begin{clm}
  \label{clm_columns}
  Assume that a layer $\mathcal S$ is covered with squares $A_1,A_2,\dots,A_s$. Fix $\gamma\in\{1,2\}$. Then there are no more than $s^2-t(r+1)$ $\gamma$-columns crossing at least $r+2$ squares in the covering.
\end{clm}

\begin{proof} Let $h$ be the number of $\gamma$-columns crossing at least $r+2$ squares in the covering. Consider the pairs of the form $(U,A_i)$, where $i\in[s]$ and $U$ is a $\gamma$-column in~$\mathcal S$ that meets $A_i$. The number of such pairs is exactly $s^2$ since every square $A_i$ meets $s$ columns. Since each $\gamma$-column crosses at least $r+1$ squares, and $h$ of them cross at least $r+2$ squares, we obtain
$$
  s^2\geq h(r+2)+(t-h)(r+1)=t(r+1)+h,
$$
so $h\leq s^2-t(r+1)$.
\end{proof}

\begin{clm}
  \label{clm_intersection}
  For every $i\in\mathbb Z/t\mathbb Z$ and every $\gamma\in\{1,2\}$, the intersection of the sets of $\gamma$-coordinates of $C_i$ and $C_{i+s}$ has cardinality at least $s-(s^2-(r+1)t)$, i.e.
  $$
   \left|\{x_{\gamma}^0(i),\ldots,x_{\gamma}^0(i)+s-1\}\cap\{x_{\gamma}^0(i+s),\ldots,x_{\gamma}^0(i+s)+s-1\}\right|\geq s-(s^2-(r+1)t).
  $$
\end{clm}

\begin{proof} 
The layer $\mathcal{S}_{s}$ meets `sub-cubes' $C_1, C_2, \ldots, C_s$, while $\mathcal{S}_{s+1}$ meets `sub-cubes' $C_2, C_3, \ldots, C_{s+1}$. Let $A_{\ell}$ be the projection of $C_{\ell}$ onto $\mathcal{S}:=\mathcal{S}_{s}$, for $\ell\in[s+1]$.

Consider the covering of $\mathcal S$ by $A_1,A_2,\dots,A_s$. By Claim~\ref{clm_columns}, among the $\gamma$-columns crossing $A_1$, there are at least $s-(s^2-t(r+1))$ ones crossing exactly $r+1$ squares in the covering. This means that each of those columns is not covered completely by $A_2,A_3,\dots,A_s$. Hence each of those columns needs to cross $A_{s+1}$, since $\mathcal S$ is covered by $A_2,A_3,\dots,A_{s+1}$ as well. This finishes the proof.
\end{proof}

\smallskip
By Claim~\ref{clm_intersection}, for each $i\in\mathbb Z/t\mathbb Z$ and $\gamma\in\{1,2\}$ we have $|x_\gamma^0(i) - x_\gamma^0(i+s)| \leq s^2 - t(r+1)$. Therefore,
$$
|x_\gamma^0(i) - x_\gamma^0(i+rs)| \leq\sum_{j=1}^r |x_\gamma^0(i+(j-1)s) - x_\gamma^0(i+js)| \leq (s^2 - t(r+1))r.
$$
Now we distinguish two cases.

\smallskip
\noindent
\textit{Case 1.} The above inequality is sometimes strict, i.e., there exist $\gamma\in\{1, 2\}$ and $i\in\mathbb Z/t\mathbb Z$ such that
$$
  |x_{\gamma}^0(i) - x_{\gamma}^0(i+rs)| < (s^2 - t(r+1))r.
$$
Without loss of generality, we assume $i=0$. For $\kappa=1,2$, we denote
$$
  N_\kappa:=\left|\{x_\kappa^0(0),\ldots,x_\kappa^0(0)+s-1\}\cap\{x_\kappa^0(rs),\ldots,x_\kappa^0(rs)+s-1\}\right|;
$$
in other words, $N_\kappa$ is the number of common $\kappa$-coordinates in $C_0$ and $C_{rs}$.
Then we have
\begin{equation}
  \label{last_intersection}
  N_\gamma>s - (s^2 - t(r+1))r \quad\text{and}\quad
  N_{3-\gamma}\geq s - (s^2 - t(r+1))r.
\end{equation}

We show that there exists $\kappa\in\{1,2\}$ such that
\begin{equation}
  N_\kappa>s(r+1)-t \quad\text{and} \quad
  N_{3-\kappa}>s^2-t(r+1).
  \label{kappa}
\end{equation}
Indeed, notice that
\begin{equation}
  \label{inter}
  s - (s^2 - t(r+1))r
  =s(r+1)-(s+s^2-t(r+1))r\geq s(r+1)-t
\end{equation}
by~\ref{5_ineq_main}; moreover, if the inequality in~\ref{5_ineq_main} is strict, then so is the last inequality above.
By~\ref{4 or 5}, one of the inequalities in~\ref{4_ineq_main} and~\ref{5_ineq_main} is strict. If the inequality in~\ref{4_ineq_main} is strict, then we can put $\kappa=\gamma$, since $N_{3-\kappa}=N_{3-\gamma}\geq s-(s^2-t(r+1))r>s^2-t(r+1)$ by~\ref{4_ineq_main} and $N_\kappa=N_\gamma>s - (s^2 - t(r+1))r\geq s(r+1)-t$ by~\eqref{inter}. Otherwise, \ref{5_ineq_main} is strict, and we can put $\kappa=3-\gamma$, as $N_\kappa=N_{3-\gamma}\geq s - (s^2 - t(r+1))r>s(t+1)-t$ by~\eqref{inter} and $N_{3-\kappa}=N_\gamma>s-(s^2-t(r+1))r\geq s^2-t(r+1)$ by~\ref{4_ineq_main}.

\smallskip
Without loss of generality, we assume that $\kappa=1$ satisfies~\eqref{kappa}. %
Consider $\mathcal{S}:=\mathcal{S}_0$, and let $A_\ell$ be the projection of $C_\ell$ onto $\mathcal S$, for $\ell\in\{t+1-s,t+2-s,\dots,t\}$. Notice that $\mathcal S$ is covered by the $A_\ell$, where $\ell$ runs through the same range.

Recall that $t-s+1\leq rs< t$ by~\ref{one}. In $\mathcal S$, consider any of $\kappa$-columns crossing both $A_{rs}$ and $A_0(=A_t)$; there are at least $N_{3-\kappa}>s^2-t(r+1)$ such $\kappa$-columns. By Claim~\ref{clm_columns}, at least one of those is covered by exactly $r+1$ squares in the covering. But the intersections of the column with two of those squares, namely $A_t$ and $A_{rs}$, have $N_\kappa$ common elements, hence those $r+1$ squares cover at most
$$
  s(r+1)-N_\kappa<s(r+1)-(s(r+1)-t)=t
$$
elements in the column. Therefore, this column is not covered completely --- a contradiction.

\smallskip
\noindent
\textit{Case 2.} Conversely, assume now that, for every $\gamma\in\{1, 2\}$ and $i\in \mathbb Z/t\mathbb Z$,
\begin{equation}
|x_\gamma^0(i) - x_\gamma^0(i+rs)| =\sum_{j=1}^r |x_\gamma^0(i+(j-1)s) - x_\gamma^0(i+js)| = (s^2 - t(r+1))r.
\label{two_equalities}
\end{equation}

Fix $\gamma\in\{1,2\}$ and $i\in\mathbb Z/t\mathbb Z$. By the triangle inequality, (\ref{two_equalities}) is only possible when, for every $j\in[r]$,
$$
  |x_\gamma^0(i+(j-1)s) - x_\gamma^0(i+js)|=s^2-t(r+1)
$$
and, moreover, all expressions of the form $x_\gamma^0(i+(j-1)s) - x_\gamma^0(i+js)$ are equal, for $j\in[r]$. This in fact means that all expressions of the form $x_\gamma^0(i)-x_\gamma^0(i+s)$ are equal, and this common value is $\pm(s^2-t(r+1))$. Without loss of generality, we assume that, for every $i\in \mathbb Z/t\mathbb Z$,
$$
x_2^0(i+s) - x_2^0(i) = x_1^0(i+s) - x_1^0(i) = s^2 - t(r + 1)
$$
(the generality is not lost, as we can renumber each coordinate as $i\mapsto -i$).

So, the first and the second coordinates of each base vertex are equal; so those coordinates of any point in a `sub-cube' differ by at most $s-1$. Therefore, no point of the form $(0,s,x_3)$ is covered by any `sub-cube' (recall here that $t>2s$ by~\ref{one}). This is a contradiction.
%
%
%
%
%
\end{proof}


\smallskip
\textit{Proof of Theorem~\ref{lower_bounds}.} We show that the parameters in the theorem satisfy \ref{one}--\ref{4 or 5}; the theorem follows then from Lemma~\ref{Lemma_for_lower_bounds}.

To prove~\ref{one}, write
$$
0<t-rs=\xi r+\lfloor\xi^2/(r+1)\rfloor\leq
\xi r+r-1\leq r^2+r-1<s.
$$

Notice that
\begin{equation}
 s^2-t(r+1)=\xi^2-\left\lfloor\frac{\xi^2}{r+1}\right\rfloor(r+1)\in[0,\xi],
\label{s^2-t(r+1)}
\end{equation}
where the last inclusion follows from~\eqref{xi^2}. Therefore,
$$
 \left\lceil\frac{t}{s}\left\lceil\frac{t}{s}\right\rceil\right\rceil=
 \left\lceil\frac{t(r+1)}{s}\right\rceil=
  \left\lceil s-\frac{s^2-t(r+1)}{s}\right\rceil\in
  \left[s-\left\lfloor\frac{\xi}{s}\right\rfloor,s\right]
$$
which yields~\ref{two}. 

Moreover,~(\ref{s^2-t(r+1)}) implies that
$$
 (s^2-t(r+1))(r+1)\leq\xi(r+1)\leq r^2+r
 <
 r^2+r+\xi=s.
$$
This verifies~\ref{4_ineq_main} with a strict inequality (hence \ref{4 or 5} as well).

Finally,~(\ref{s^2-t(r+1)}) implies that
$$
 (s + s^2 - t(r+1))r\leq(s+\xi)r=(r^2 + r + 2\xi)r=r^3 + r^2+ 2r\xi\leq t
$$
which establishes~\ref{5_ineq_main}.
\qed

\medskip
\textit{Proof of Theorem~\ref{right_interval}.} For convenience, we replace $r$ with $r+1$ (i.e., we prove that, for any $r\in\mathbb{N}$ and any $\varepsilon\in\left[\frac{1}{r+1},\frac{(r+1)^2-1}{(r+1)^3-r-2}\right)$, $\mu(3;\varepsilon)=(r+1)^3$).

Let $s = (r+1)^2$, $t = (r+1)^3 - 1$. Notice that the fractions $\frac{1}{r+1}$ and $\frac{(r+1)^2-1}{(r+1)^3-r-2}$ are Farey neighbours (for properties of Farey sequence, see \cite[Chapter~III]{hardy-wright}). Therefore, the rational number between them with the smallest denominator is $\frac{1+(r+1)^2-1}{(r+1)+(r+1)^3-r-2}=\frac{(r+1)^2}{(r+1)^3-1}$, and the next smallest denominator is already greater than $(r+1)^3$. By Lemma~\ref{lm2}, in order to prove Theorem~\ref{right_interval}, it suffices to show that $\mu(3;s/t)\geq (r+1)^3$. For that purpose, we show that the numbers $s$ and $t$ satisfy the conditions of Lemma~\ref{Lemma_for_lower_bounds}.

%
%
Clearly, $s$ and $t$ are coprime. The conditions \ref{one}--\ref{two} 
are straightforward. Next, since $s^2-t(r+1)=r+1$, we have
$$
(s^2 - t(r+1))r = (r+1)r = (r+1)^2 - (r+1) = s - (s^2- t(r+1))
$$
which yields~\ref{4_ineq_main} (with the equality sign). It remains to show~\ref{5_ineq_main} with the strict sign:
$$
  r(s^2 + s - t(r+1)) = r(r+1)(r+2) = (r+1)((r+1)^2 - 1) < t.
  \eqno\Box
$$

\section{Proof of Theorem~\ref{upper_bounds}}
\label{th5_proof}

Let $r\geq 2$ be an integer, $\xi\in[r]$, $s=r^2+r+\xi$, $t=r^3+r^2+\xi(r+1)$.
We construct a covering of $[\mathbb{Z}/t\mathbb{Z}]^3$ consisting of $t$ `sub-cubes' of side $s$.

The main idea here is the same as in~\cite{2-dimension}: to take each next `sub-cube' shifted relative to the previous one by a fixed integer vector $v$. We take
$$
v = (1, r, (r^2+r+1)).
$$
Thus, we need to prove that the `sub-cubes' $C_i$ with base vertices $(i, r i, (r^2+r+1)i)$, $i \in \mathbb Z/t\mathbb Z$, cover the torus $[\mathbb{Z}/t\mathbb{Z}]^3$. Since the points of the lattice with the same first coordinate $j$ are covered by $s$ consecutive cubes $C_{j-i}$ with $i\in\{0,1,\dots,s-1\}$, the following claim finishes the proof of Theorem~\ref{upper_bounds}.



\begin{clm} The squares of side $s$ with base vertices $(r i, (r^2+r+1)i)$, $i \in \{0,1, \ldots, s-1\}$, cover $[\mathbb{Z}/t\mathbb{Z}]^2$.
\label{clm 2-dimension}
\end{clm}

\begin{proof} 
Denote by $B_i$, $i\in\{0,1,\ldots,s-1\}$, the square with the base vertex $(ri, (r^2 + r + 1)i)$.
In what follows, we assume that $i$ runs over $\mathbb Z/s\mathbb Z$, i.e., $i+1=0$ for $i=s-1$.  

Fix $y\in\{0,1,\dots,t-1\}$. Let $C$ be the column consisting of all points whose first coordinate is~$y$. We show that $C$ is covered completely by the $B_i$.

Notice that the relative shift of the first coordinate of $B_{i+1}$ with respect to $B_i$ is
$x_1^0(i+1) - x_1^0(i) = r$ for $i\neq s-1$ and
$$
  x_1^0(0) - x_1^0(s-1) = t - r(s-1)
  = r+\xi
$$
for $i=s-1$. Since $r+\xi<s$, column $C$ meets several consecutive squares $B_i$, $B_{i+1}$, \dots, $B_{i+p}$. Since each of $B_{i-1}$ and $B_{i+p+1}$ does not meet~$C$, we have
$x_1^0(i+p+1)-x_1^0(i-1)>s$. Therefore,
$$
  r^2+r+\xi+1=s+1\leq x_1^0(i+p+1)-x_1^0(i-1)\leq (p+1)r+(r+\xi),
$$
which yields $p\geq r$. So, $C$ meets the squares $B_i,B_{i+1},\dots,B_{i+r}$.

Now, the relative shift of the second coordinate of $B_{i+1}$ with respect to $B_i$ is $x_2^0(i+1) - x_2^0(i) = r^2+r+1$ for $i\neq s-1$ and
$$
  x_2^0(0)-x_2^0(s-1)=t(r+1) - (r^2+r+1)(s-1) = \xi r + 1< r^2+r+1
$$
for $i=s-1$. This yields that the parts of $C$ covered with consecutive squares are either tangent or even overlapping. So the interval they cover consists of at least $\min\{t,s+(r-1)(r^2+r+1)+(r\xi+1)\}=t$ points, and hence the whole column is covered.
\end{proof}

\section{Proof of Theorem~\ref{first_interval}}
\label{th1_proof}

We start with the lower bounds for $\mu(3;\varepsilon)$ agreeing with the statement of the theorem. For $\varepsilon\in[1/2,4/7)\cup[3/5,1)$, such lower bounds follow from~\eqref{known-bounds}. For $\varepsilon\in[4/7,3/5)$ the lower bound follows from Theorem~\ref{right_interval}.

In order to show that those bounds are achieved, it suffices to show that $\mu_0(3;3,4)=4$, $\mu_0(3;2,3)=5$, and $\mu_0(3;3,5)=7$. The corresponding examples are provided by the following sets of base vertices:
\begin{align*}
  &(0, 0, 0),\; (1, 1, 1),\;  (2, 2, 2),\; (3, 3, 3), &&\text{for $s/t=3/4$;}\\
  &(0, 0, 0), \; (1, 1, 1),\; (1, 2, 2),\; (2, 1, 2),\; (2, 2, 1), &&\text{for $s/t=2/3$;}\\
  &(0,0,0),\; (1,1,3),\; (1,3,1),\; (2,4,4),\; (3,0,2),\; (3,2,0),\; (4,3,3), &&\text{for $s/t=3/5$.}
\end{align*}
The theorem is proved.

\section{Acknowledgements}

The study was supported by the Russian Science Foundation (grant number 21-71-10092).

\end{document}